\documentclass[10pt,a4paper,reqno]{amsart}
\usepackage{amsfonts,amsthm,latexsym,amsmath,amssymb,amscd,amsmath,epsf, 
graphicx}

\newtheorem{tm}{Theorem}

\newtheorem{rem}[tm]{Remark}

\newtheorem{lm}[tm]{Lemma}

\newtheorem{prop}[tm]{Proposition}

\newtheorem{problem}[tm]{Problem}
\newtheorem{conj}[tm]{Conjecture}

\newcommand{\si}{\sigma}

\newcommand{\bZ}{\mathbb Z}
\newcommand{\D}{\mathcal D}

\newcommand{\bsi}{\overline \sigma}
\newcommand{\tsi}{\widehat \si}

\begin{document}
\title[Something You Always Wanted to Know About Real Polynomials] {Something You Always Wanted to Know About Real Polynomials (But Were Afraid to Ask)}

\author[V.~Kostov]{Vladimir P. Kostov }
\noindent 
\address{Universit\'e C\^ote d'Azur, CNRS, LJAD, France
}
             \email{kostov@math.unice.fr}

\author[B.~Shapiro]{Boris Z. Shapiro}
\noindent 
\address{Department of Mathematics, Stockholm University, SE-106 91
Stockholm,
         Sweden}
\email{shapiro@math.su.se}



\begin{abstract}
The famous Descartes' rule of signs from 1637 giving an upper bound on the number of positive roots of a real univariate polynomials in terms of the number of sign changes of its coefficients, has been an indispensable source of inspiration for generations of mathematicians.  Trying to extend   and sharpen this rule,   we consider below the set of all real univariate polynomials of a given degree, a given collection of  signs of their coefficients,  and a given number of positive and negative roots.  In spite of the elementary definition of the main object of our study, it is a non-trivial question for which sign patterns and numbers of positive and negative roots the corresponding set is non-empty. The main result of the present paper is a discovery of a new infinite family of non-realizable combinations of sign patterns and the numbers of positive and negative roots.   
\end{abstract}

\date{}
\maketitle

\section{Introduction}

This paper\footnote{The title of the present paper alludes  to one of the funniest movies by Heywood ``Woody" Allen, the favorite movie director of the second author.} continues the line of study of   Descartes' rule of signs initiated in   \cite{FoKoSh}.  The basic set-up under consideration is as follows. 

\medskip
Consider the affine space $Pol_d$ of all real monic univariate polynomials of degree $d$. % and define  the {\it standard real discriminant locus}   $\D_d\subset Pol_d$ as the subset of all polynomials having a real multiple root.  (Detailed information about a natural stratification of $\D_d$ can be found in e.g., \cite {KhS}.)  It is a  well-known and simple fact that $Pol_d\setminus \D_d$ consists of $\left[\frac{d}{2}\right]+1$ components distinguished  by the number of real  simple roots. Moreover, each such component is contractible in $Pol_d$.  
Below we  concentrate on polynomials from $Pol_d$ with all non-vanishing coefficients. An arbitrary ordered sequence 
$\bsi=(\si_0,\si_1,\dots,\si_d)$ of $\pm$-signs  is called a {\em sign pattern}. When working with monic polynomials we will  use their {\em shortened sign patterns} $\tsi$ representing the signs of all coefficients except the leading term which equals $1$. For the actual sign pattern $\bsi,$ we  write $\bsi=(1,\tsi)$ to emphasise that we consider  monic polynomials.

%  For any pair $(d,k)$  of non-negative integers with $d-k\ge 0;\; d-k\equiv 0\text{\; {\rm mod\; 2}},$ denote by $Pol_{d,k}$,
%the set of all monic real polynomials of degree $d$ with $k$ real  simple roots.  (As we mentioned above, each $Pol_{d,k}$ is an open contractible subset of $Pol_d$.)  
%Denote by $Pol_d({\tsi})\subset Pol_d$ the orthant of all monic polynomials $p=x^d+a_1x^{d-1}+\dots +a_d$ whose coefficients $(a_1,\dots,a_d)$ have the shortened sign pattern $\tsi=(\si_1,\dots,\si_d)$. Finally, set 
%$$Pol_{d,k}(\tsi):=Pol_{d,k}\cap Pol_d(\tsi).$$ 
%Although it is  difficult to be completely sure about such an elementary topic, to the best of our knowledge,  for arbitrary $(d, k)$ and $\tsi$, the problems of non-emptiness and/or topology of general $Pol_{d,k}(\tsi)$  have only been considered  in \cite{AlFu}, \cite{AJS},  \cite{Gr} and  \cite{FoKoSh}. 

Given a shortened sign pattern $\tsi$, we call by its {\em Descartes' pair} $(p_{\tsi}, n_{\tsi})$ the pair of non-negative integers counting sign changes and sign preservations of $\bsi=(1,\tsi)$.  By Descartes' rule of signs,  $p_{\tsi}$ (resp. $n_{\tsi}$) gives the upper bound on the number of positive (resp. negative)   roots  of any monic polynomial from $Pol_d({\tsi})$. (Observe that, for any $\tsi$, 
$p_{\tsi}+ n_{\tsi}=d$.) To any monic polynomial $q(x)$ with the sign pattern $\bsi=(1,\tsi),$ we associate  the pair $(pos_q, neg_q)$ giving the numbers of its positive and negative roots counted with multiplicities. Obviously the pair $(pos_q,neg_q)$  satisfies the standard restrictions
 \begin{equation}\label{stand} 
pos_q\le  p_{\bsi}, \;  pos_q\equiv p_{\bsi} ({\rm mod\, 2}),\; neg_q\le n_{\bsi},\; neg_q\equiv n_{\bsi} ({\rm mod\, 2}).
\end{equation}

We call pairs $(pos, neg)$ satisfying \eqref{stand} {\em admissible} for $\bsi$.  Conversely, for a given pair  $(pos, neg),$ we call a sign pattern $\bsi$ such that \eqref{stand} is satisfied {\em admitting} the latter pair. 
 It turns out that not for every pattern $\bsi,$ all its admissible  pairs  $(pos,neg)$  are realizable by polynomials with the sign pattern $\bsi$. Namely,  D.~J.~Grabiner \cite{Gr} found the first example of non-realizable combination for polynomials of degree $4$. He has shown that the sign pattern $(+,-,-,-,+)$ does not allow to realize the pair $(0,2)$ and the sign pattern  $(+,+,-,+,+)$ does not allow to realize $(2,0)$. Observe that their Descartes' pairs equal $(2,2).$  

His argument is very simple. (Due to symmetry induced by $x\mapsto -x$ it suffices to consider only the first case.) Observe that a fourth-degree polynomial with only two negative  roots for which the sum of roots is positive could be factored as $a(x^2+bx+c)(x^2-sx+t)$ with  $a,b,c,s,t>0,\; s^2<4t$, and $b^2\ge 4c$. The product of these factors equals $a(x^4+(b-s)x^3+(t+c-bs)x^2+(bt-cs)x+ct)$.
To get the correct sign pattern, we need $b<s$ and $bt<cs,$ which gives $b^2t<s^2c$ and thus $b^2/c<s^2/t$. But we have $b^2/c\ge 4>s^2/t$.   

\medskip
The following basic question and related conjecture were formulated in \cite{FoKoSh}. (Apparently for the first time Problem~\ref{prob:basic}Ê   was mentioned in \cite{AJS}.) 

\begin{problem}\label{prob:basic}
For a given sign pattern $\bsi,$  which admissible pairs $(pos,neg)$ are realizable by polynomials whose signs of coefficients are given by  $\bsi$?  
\end{problem}

\medskip
Observe that we have the natural $\bZ_2\times \bZ_2$-action   on the space of monic polynomials and on the set of all sign patterns respectively.  The first generator acts by reverting the signs of all monomials in second, 
fourth etc. position 
(which for polynomials means $P(x)\to (-1)^dP(-x)$); 
the second generator acts by reading the pattern backwards 
(which for polynomials means $P(x)\to x^dP(1/x)$). If one wants to  preserve the set of monic polynomials one has to divide $x^dP(1/x)$ by its leading term.  We will refer to the latter action as the  {\em standard $\bZ_2\times \bZ_2$-action}.  % WHAT ABOUT CHAGING PLUSES TO MINUSES AND VICE VERSA?
(Up to some trivialities) the properties we will study below are invariant  under this action.  
The following initial results were partially proven in \cite{AJS, AlFu} and in complete generality in \cite{FoKoSh}.

 \begin{tm}\label{th:AlFu}
 {\rm (i)} Up to degree $d\le 3,$ for any sign pattern $\bsi$, all admissible pairs 
 $(pos,neg)$ are realizable.

\noindent 
{\rm (ii)} For $d=4,$ (up to the standard $\bZ_2\times \bZ_2$-action) the only 
 non-realizable combination is $(1,-,-,-,+)$ with the pair $(0,2)$;  
 
 \noindent
 {\rm (iii)} For $d=5,$ (up to the standard $\bZ_2\times \bZ_2$-action) the only 
 non-realizable combination is $(1,-,-,-,-,+)$ with the pair $(0,3)$;
 
  \noindent
 {\rm (iv)} For $d=6,$ (up to the standard $\bZ_2\times \bZ_2$-action) the only 
 non-realizable combinations are $(1,-,-,-,-,-,+)$ with $(0,2)$ and $(0,4)$; 
 $(1,+,+,+,-,+,+)$ with $(2,0)$; 
 $(1,+,-,-,-,-,+)$ with $(0,4)$.
  
 \end{tm}

The next two results can be found in \cite{FoKoSh} and \cite{Ko2}. 

%Trying to extend Theorem~\ref{th:AlFu}, we obtained a computer-aided  classification of all non-realizable  sign patterns and pairs for $d=7$ and almost all for $d=8$, see below. 

\begin{tm}\label{th:7} For $d=7$,  among the 1472 possible combinations of a sign pattern
and a pair (up to the standard $\bZ_2\times \bZ_2$-action), there exist exactly 6 which are non-realizable. They are:
$$(1,+,-,-,-,-,-,+) \quad\text{with} \quad(0,5);\quad  (1,+,-,-,-,-,+,+) \quad\text{with} \quad(0,5);$$
$$(1,+,-,+,-,-,-,-) \quad\text{with} \quad(3,0);\quad  (1,+,+,-,-,-,-,+) \quad\text{with} \quad(0,5);$$
$$\text{and,}\quad(1,-,-,-,-,-,-,+) \quad\text{with}\;  (0,3)\;\text{and} \;(0,5).$$
\end{tm}

\begin{tm}\label{th:8}  For $d=8,$  among the 3648 possible combinations of a sign pattern and a pair (up to the standard $\bZ_2\times \bZ_2$-action), there exist exactly 13 which are non-realizable. They are:
$$(1,+,-,-,-,-,-,+,+) \quad\text{with} \quad(0,6);\quad  (1,-,-,-,-,-,-,+,+) \quad\text{with} \quad(0,6);$$
$$(1,+,+,+,-,-,-,-,+) \quad\text{with} \quad(0,6);\quad  (1,+,+,-,-,-,-,-,+) \quad\text{with} \quad(0,6);$$
$$(1,+,+,+,-,+,+,+,+) \quad\text{with} \quad(2,0);\quad  (1,+,+,+,+,+,-,+,+) \quad\text{with} \quad(2,0);$$
$$(1,+,+,+,-,+,-,+,+) \quad\text{with}\;  (2,0)\;\text{and} \;(4,0)\;;\quad(1,-,-,-,+,-,-,-,+) \quad\text{with}\;$$  
$$ (0,2)\;\text{and} \;(0,4); \quad(1,-,-,-,-,-,-,-,+) \quad\text{with}\;  (0,2), (0,4), \text{  and } (0,6).$$
\end{tm}
%\begin{rem}\label{rem:8} {\rm For $d=8$, there exist exactly 7 (up to the standard $\bZ_2\times \bZ_2$-action) 
%combinations of a sign pattern and a pair for which it is still unknown whether they
%are realizable or not. They are:
%$$(1,+,-,+,-,-,-,+,+) \quad\text{with} \quad(4,0);\quad  (1,+,-,+,-,+,-,-,+) \quad\text{with} \quad(4,0);$$
%$$(1,+,+,-,-,-,-,+,+) \quad\text{with} \quad(0,6);\quad  (1,+,+,-,-,+,-,+,+) \quad\text{with} \quad(4,0);$$
%$$(1,+,+,+,-,+,-,-,+) \quad\text{with} \quad(4,0); 
%\quad(1,+,-,+,-,-,-,-,+) \quad\text{with}\;  (4,0)$$ $$\;\text{and} \;(0,4).$$}
%\end{rem}

Based on Theorems~\ref{th:AlFu} -- \ref {th:8}, we formulated in \cite {FoKoSh} the following guess. 
\medskip

\begin{conj}\label{conj:main} For an arbitrary sign pattern $\bsi$, the only type of pairs $(pos,neg)$ which can be non-realizable has either $pos$ or $neg$ vanishing. In other words, for any sign pattern $\bsi$, each pair $(pos,neg)$ satisfying \eqref{stand}  with positive $pos$ and $neg$ is realizable. 
\end{conj}

%Rephrasing the above conjecture, we say that the only phenomenon implying non-realisability is that "real roots on one half-axis  force real roots on the other half-axis``.  
At the moment Conjecture~\ref{conj:main} has been verified by computer-aided methods up to $d=10$. 
The main result of the present paper  is a discovery of a new  infinite series of non-realizable patterns which supports Conjecture~\ref{conj:main}. (Two other series can be found in \cite{FoKoSh}.)  
  Namely, for a fixed odd degree $d\geq 5$ and $1\leq k\leq (d-3)/2$, denote by $\sigma _k$ the sign pattern 
beginning with two pluses followed by $k$ pairs $``-,+"$ and then by 
$d-2k-1$ minuses. Its Descartes' pair equals $(2k+1,d-2k-1)$.

\begin{tm} \label{th:series}
{\rm(i)} The sign pattern $\sigma _k$ is not realizable with any of 
the pairs $(3,0)$, 
$(5,0)$, $\ldots$, $(2k+1,0)$; 

\noindent 
{\rm (ii)} the sign pattern $\sigma _k$ is realizable with the pair $(1,0)$;

\noindent 
{\rm (iii)} the sign pattern $\sigma _k$ is realizable with any of the pairs 
$(2\ell +1,2r)$, $\ell =0$, $1$, $\ldots$, $k$, $r=1$, $2$, $\ldots$, 
$(d-2k-1)/2$.
\end{tm} 

Notice that Cases (i), (ii) and (iii) exhaust all possible admissible pairs $(pos,neg)$. It is also worth mentioning that the only non-realizable case for degree 5
 (up to the $\bZ_2\times \bZ_2$-action) and the third and the last two non-realizable cases for degree 7 mentioned above 
are covered by Theorem~\ref{th:series}.

\medskip
The structure of the paper is as follows. In \S~\ref{sec:proofs} we present a proof of Theorem~\ref{th:series}. In \S~\ref{sec:zoo} we present the detailed structure of the discriminant loci and (non)realizable patterns for polynomials of degrees 3 and 4.

\medskip 
\noindent 
{\em Acknowledgements.}   The first author is grateful to the Mathematics Department of Stockholm University for the hospitality.  

\section{Proofs}\label{sec:proofs}

\begin{proof}[Proof of Theorem~\ref{th:series}]

Part (i): Suppose that a polynomial $P:=\sum _{j=0}^da_jx^{d-j}$ has the sign pattern 
$\sigma _k$ and realizes the pair $(2s+1,0)$, $1\leq s\leq k$. 
Denote by  
$$P_e:=\sum _{\nu =0}^{(d-1)/2}a_{2\nu +1}x^{d-2\nu -1}~~~{\rm and}~~~ 
P_o:=\sum _{\nu =0}^{(d-1)/2}a_{2\nu}x^{d-2\nu}$$ 
its even and odd parts respectively. In each of 
the sequences $\{ a_{2\nu +1}\} _{\nu =0}^{(d-1)/2}$ and 
$\{ a_{2\nu}\} _{\nu =0}^{(d-1)/2}$ there is exactly one sign change. Therefore 
each of the polynomials $P_e$ and $P_o$ has exactly one real positive root 
(denoted by $x_e$ and $x_o$ respectively) which is simple. The polynomial 
$P_e$ (resp. $P_o$) is positive and increasing on $(x_e,\infty )$ (resp. 
on $(x_o,\infty )$) and negative on $[0,x_e)$ (resp. on $(0,x_o)$). 

The polynomial $P$ has at least three distinct positive roots. Denote the 
smallest of them by 
$0<\xi _1<\xi _2<\xi _3$. Hence at any point $\zeta \in (\xi _1,\xi _2)$ 
one has the $P(\zeta )>0$; clearly $P$ is negative on $(\xi _2,\xi _3)$. 
One can choose $\zeta \neq x_e$ and 
$\zeta \neq x_o$. Hence it is 
impossible to have $P_e(\zeta )<0$ and $P_o(\zeta )<0$. It is also 
impossible to have $P_e(\zeta )>0$ and $P_o(\zeta )>0$. Indeed, 
this would imply that $x_e<\zeta$ and $x_o<\zeta$. Thus one would get  
$P_e(x)>0$ and $P_o(x)>0$, i.e. $P(x)>0$, 
for $x\in (\xi _2,\xi _3)$ -- a contradiction.

The  two remaining possibilities are: 

\noindent
{\rm a)} $P_e(\zeta )>0$, $P_o(\zeta )<0$; 

\noindent
{\rm b)} $P_e(\zeta )<0$, $P_o(\zeta )>0$. 

\medskip
The first one is impossible because it would imply that  
$$P(-\zeta )=P_e(\zeta )-P_o(\zeta )>0~,$$
and since $P(0)<0$ and $P(x)\rightarrow -\infty$ for $x\rightarrow -\infty$, the 
polynomial $P$ would have at least one negative root in $(-\infty ,-\zeta )$ 
and at least one in $(-\zeta ,0)$ -- a contradiction. 

So suppose that possibility {\rm b)}  takes place.  
In this case one must have $x_o<\zeta <x_e$. Without loss of generality one 
can assume that $\xi _1=1$; this can be achieved by a rescaling  
$x\mapsto \alpha x$ with $\alpha >0$.  
Hence $P_o(1)=\beta >0$ and $P_e(1)=-\beta$. Considering the polynomial $P/\beta$ instead of $P,$  
 one can assume that $\beta =1$. Lemma~\ref{lmlm}  below immediately implies that there are no real roots of $P$ larger than $1$ which 
is a contradiction finishing the proof of Part (i). 

\begin{lm}\label{lmlm}
Under the above assumptions,   $P^{(m)}(1)>0$, for any $m=1, 2, \ldots, d$.
\end{lm}

\begin{proof}[Proof of Lemma~\ref{lmlm}.]
For any $m=1$, $2$, $\ldots$, $d,$ it is true that if the sum of the 
coefficients $\delta :=a_2+a_4+\cdots +a_{d-1}$ is fixed (recall that all these 
coefficients are negative), then $P_o^{(m)}(1)$ is minimal for $a_2=\delta$, 
$a_4=a_6=\cdots =a_{d-1}=0$. Indeed, when taking derivatives and  
computing their values at $x=1$, the monomial with the largest degree in $x$ 
is multiplied by the largest  factor (equal to this degree). Therefore 
in what follows we assume that $a_4=a_6=\cdots =a_{d-1}=0$, and  hence $a_2=1-a_0<0$. 

Similarly, consider  $P_e^{(m)}(1)$. Recall that $a_1>0$, $a_3>0$, $\ldots$, 
$a_{2k+1}>0$, $a_{2k+3}<0$, $a_{2k+5}<0$, $\ldots$, $a_d<0$. Hence for fixed sums 
$\delta _*:=a_1+a_3+\cdots +a_{2k+1}$ and 
$\delta _{**}:=a_{2k+3}+a_{2k+5}+\cdots +a_d$, the value of $P_e^{(m)}(1)$ is 
minimal if 

\begin{equation}\label{eqeq}
\begin{cases}%\begin{array}{cccccccccccc}
a_1=\cdots=a_{2k-1}=0~,~a_{2k+1}=\delta _*\\  
a_{2k+5}=\cdots=a_d=0~,~a_{2k+3}=\delta _{**}.
\end{cases}%\end{array}
\end{equation} 

\medskip
Let us now assume that conditions (\ref{eqeq}) are valid. Thus 
$P_e=a_{2k+1}x^{d-2k-1}+a_{2k+3}x^{d-2k-3}$ and $a_{2k+1}+a_{2k+3}=-1$. One can further 
decrease $P_e^{(m)}(1)$ by assuming that $a_{2k+1}=0$, $a_{2k+3}=-1$. Thus 
$P(x)=a_0x^d+a_2x^{d-2}-x^{d-2k-3}~~~{\rm and}~~~a_0+a_2=1~.$

But then $P^{(m)}(x)=u_ma_0x^{d-m}+v_ma_2x^{d-2-m}-w_mx^{d-2k-3-m}$ and 
$P^{(m)}(1)=u_ma_0+v_ma_2-w_m$ for some numbers $0\leq w_m\leq v_m<u_m$. Therefore 
$$\begin{array}{ccccc}
P^{(m)}(1)&=&w_m(a_0+a_2-1)+(v_m-w_m)(a_0+a_2)+(u_m-v_m)a_0&&\\ &=&
(v_m-w_m)(a_0+a_2)+(u_m-v_m)a_0>0~.&&\end{array}$$
  \end{proof}

\medskip 
Proof of Part (ii): 
The polynomial $x^d-1$ has the necessary signs of the leading coefficient and 
of the constant term. It has a single real simple root at $1$. One can 
construct a polynomial of the form 
$S:=x^d-1+\varepsilon \sum _{j=1}^{d-1}c_jx^j$, where $c_j=1$ (resp. $c_j=-1$) if 
the sign at the corresponding position of $\sigma _k$  is  $+$ (resp.  $-$). 
For a small enough $\varepsilon >0$,  the polynomial $S$ has a single simple real root 
  close to $1$, and its coefficients  
have the sign pattern $\sigma$.

Finally, our approach how to settle  Part (iii) is based on the following 
lemma borrowed from \cite{FoKoSh}. For a monic polynomial we might write $1$ instead of the 
first $+$ sign in its sign pattern. Recall that the shortened sign pattern of a monic 
polynomials is what remains from its sign pattern when this initial $1$ is 
deleted.  

\begin{lm}[See Lemma~14 in \cite{FoKoSh}] \label{lmconcat} Suppose that the
monic polynomials $P_1$ and $P_2$ of degrees $d_1$ and $d_2$ with sign
patterns $\bar{\sigma}_1 = (1,\hat{\sigma}_1)$ and 
$\bar{\sigma}_2 = (1,\hat{\sigma}_2)$, 
respectively, realize
the pairs $(pos_1, neg_1)$ and $(pos_2, neg_2)$. %(Here $\hat{\sigma}_1$ and 
%$\hat{\sigma}_2$ are the shortened sign patterns of $P_1$ and $P_2$ respectively.) 

\medskip
Then

\noindent
{\rm (i)} if the last position of $\hat{\sigma}_1$ is $+$, then for any 
small enough $\varepsilon > 0$, the polynomial $\varepsilon ^{d_2}P_1(x)P_2(x/\varepsilon )$ 
realizes the sign pattern $(1,\hat{\sigma}_1,\hat{\sigma}_2)$ and the pair 
$(pos_1+pos_2, neg_1+neg_2)$.

\noindent
{\rm (ii)} if the last position of $\hat{\sigma}_1$ is $-$, then for any 
$\varepsilon > 0$
small enough, the polynomial $\varepsilon ^{d_2}P_1(x)P_2(x/\varepsilon )$ 
realizes
the sign pattern $(1,\hat{\sigma}_1,−\hat{\sigma}_2)$ and the pair 
$(pos_1+pos_2, neg_1+neg_2)$. (Here $-\hat{\sigma}$ is the sign pattern obtained
from $\hat{\sigma}$ by changing each $+$ by $-$ and vice versa.)
\end{lm}

\begin{rem}\label{remconcat}
{\rm Example~15 in \cite{FoKoSh} explains some of the 
possible applications  of Lemma~\ref{lmconcat}. We present and extend 
this example below. 
If} 
$$P_2=x-1~,~x+1~,~x^2+2x+2~,~x^2+2x+0.5~,~x^2-2x+2~~~{\rm or}~~~x^2-2x+0.5~,$$ 
{\rm then $(pos_2,neg_2)=(1,0)$, $(0,1)$, $(0,0)$, $(0,2)$, $(0,0)$ and $(2,0)$ 
respectively. Denote 
by $\tau$ the last entry of $\hat{\sigma}_1$. When $\tau =+$, then one has 
respectively $\hat{\sigma}_2=(-)$, $(+)$, $(+,+)$, $(+,+)$, $(-,+)$ and 
$(-,+)$ and the 
sign pattern of $\varepsilon ^{d_2}P_1(x)P_2(x/\varepsilon )$ equals} 
$$(1,\hat{\sigma}_1,-)~,~(1,\hat{\sigma}_1,+)~,~(1,\hat{\sigma}_1,+,+)
~,~(1,\hat{\sigma}_1,+,+)~,~(1,\hat{\sigma}_1,-,+)~~~
{\rm or}~~~(1,\hat{\sigma}_1,-,+)~.$$
{\rm If $\tau =-$, then $\hat{\sigma}_2=(+)$, $(-)$, $(-,-)$, $(-,-)$, 
$(+,-)$ and $(+,-)$ 
and the 
sign pattern of $\varepsilon ^{d_2}P_1(x)P_2(x/\varepsilon )$ equals}
$$(1,\hat{\sigma}_1,+)~,~(1,\hat{\sigma}_1,-)~,~(1,\hat{\sigma}_1,-,-)~,~
(1,\hat{\sigma}_1,-,-)~,~(1,\hat{\sigma}_1,+,-)~~~
{\rm or}~~~(1,\hat{\sigma}_1,+,-)~.$$
\end{rem}

Proof of Part (iii): 
Recall that the sign pattern $\sigma _k$ ends with $d-2k-1$ minuses. Set 
$\sigma _k=(+,+,\sigma ^*,\sigma ^{\dagger})$, where the sign patterns 
$\sigma ^*$ (resp. $\sigma ^{\dagger}$) consist of a minus followed by 
$k$ pairs $(+,-)$  
(resp. of $d-2k-2$ minuses). 

The sign pattern $(+,+)$ is realizable by the polynomial $x+1$ (hence 
with the pair $(0,1)$). To obtain a 
polynomial realizing the sign pattern $(+,+,\sigma ^*)$ with the pair 
$(2\ell +1,1)$ one applies Lemma~\ref{lmconcat}, first $k-\ell$ times with 
$P_2=x^2-2x+2$, and then $2\ell +1$ times with $P_2=x-1$. 
After this one applies 
Lemma~\ref{lmconcat}, first $2r-1$ times with $P_2=x+1$, and then 
$(d-2k-1)/2-r$ times with 
$P_2=x^2+2x+2$ to realize the sign pattern $\sigma _k$ with the pair 
$(2\ell +1,2r)$.
\end{proof}

\section{Discriminant loci of cubic and quartic polynomials under a microscope}\label{sec:zoo}

The goal of this  section is mainly pedagogical.   
For the convenience of our readers, we present below detailed descriptions and illustrations of  cases of (non)realizability of 
sign patterns and admissible pairs for polynomials of degree up to $4$.  

Define  the {\it standard real discriminant locus}   $\D_d\subset Pol_d$ as the subset of all polynomials having a real multiple root.  (Detailed information about a natural stratification of $\D_d$ can be found in e.g., \cite {KhS}.)  It is a  well-known and simple fact that $Pol_d\setminus \D_d$ consists of $\left[\frac{d}{2}\right]+1$ components distinguished  by the number of real  simple roots. Moreover, each such component is contractible in $Pol_d$. Obviously, the number of real roots in a family of monic polynomials changes if and only if  this family crosses the discriminant locus $\D_d$.

\begin{figure}[htbp]
%\includegraphics[scale=0.5]{parthetanegfirstfour.eps}
%\centerline{\hbox{\includegraphics[scale=0.7]{parthetanegfirstfour.eps}}}
\centerline{\hbox{\includegraphics[scale=0.5]{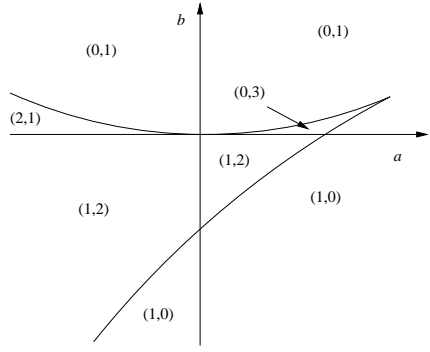}}}
%\centerline{\hbox{\epsfxsize=10cm \epsfbox{k=1234.pdf}}}
    \caption{The discriminant locus of the family $x^3+x^2+ax+b$.}
\label{n=3B}
\end{figure}

\subsection {Degrees 1 and 2}  Clearly,  a polynomial $x+u$ has a single real root $-u$ whose sign 
is opposite to the sign of the constant term. For degrees $2$, $3$ 
and $4$ we will use the invariance of the zero set of the family of 
polynomials $x^n+a_1x^{n-1}+\cdots +a_n$ with respect to  the group of 
quasi-homogeneous dilatations $x\mapsto tx$, $a_j\mapsto t^ja_j$, 
to set the subdominant coefficient to $1$. Thus for $n=2$,  
we consider the family $P_2:=x^2+x+a$.  For 
$a\leq 1/4$, it has two real roots; for $a<1/4$, these are distinct. For $a\in (0,1/4)$,  
they are both negative while for $a<0$, they are of opposite signs. 

\subsection {Degree 3} For $n=3,$ we consider the family $P_3:=x^3+x^2+ax+b$. 
Its discriminant locus $\Sigma$ 
is defined by the equation $4a^3-a^2+4b-18ab+27b^2=0$. This is a curve shown in Fig.~\ref{n=3B}. 
It has an ordinary cusp for $(a,b)=(1/3,1/27)$ and an ordinary 
tangency to the 
$a$-axis at the origin.   In the eight regions of the complement to its union with the  
coordinate axes, the polynomial has roots as indicated in 
Fig.~\ref{n=3B}.  (Here $(0,1)$ means $0$ positive and $1$ negative real roots 
hence there exists a complex conjugate pair as well.) The point of the cusp corresponds to a 
triple root at $-1/3$, the upper arc corresponds to the case 
of one double real root to the right and a simple one to the left 
(and vice versa for the lower arc).

\begin{figure}[htbp]
%\includegraphics[scale=0.5]{parthetanegfirstfour.eps}
%\centerline{\hbox{\includegraphics[scale=0.7]{parthetanegfirstfour.eps}}}
\vskip0.5cm
\centerline{\hbox{\includegraphics[scale=0.2]{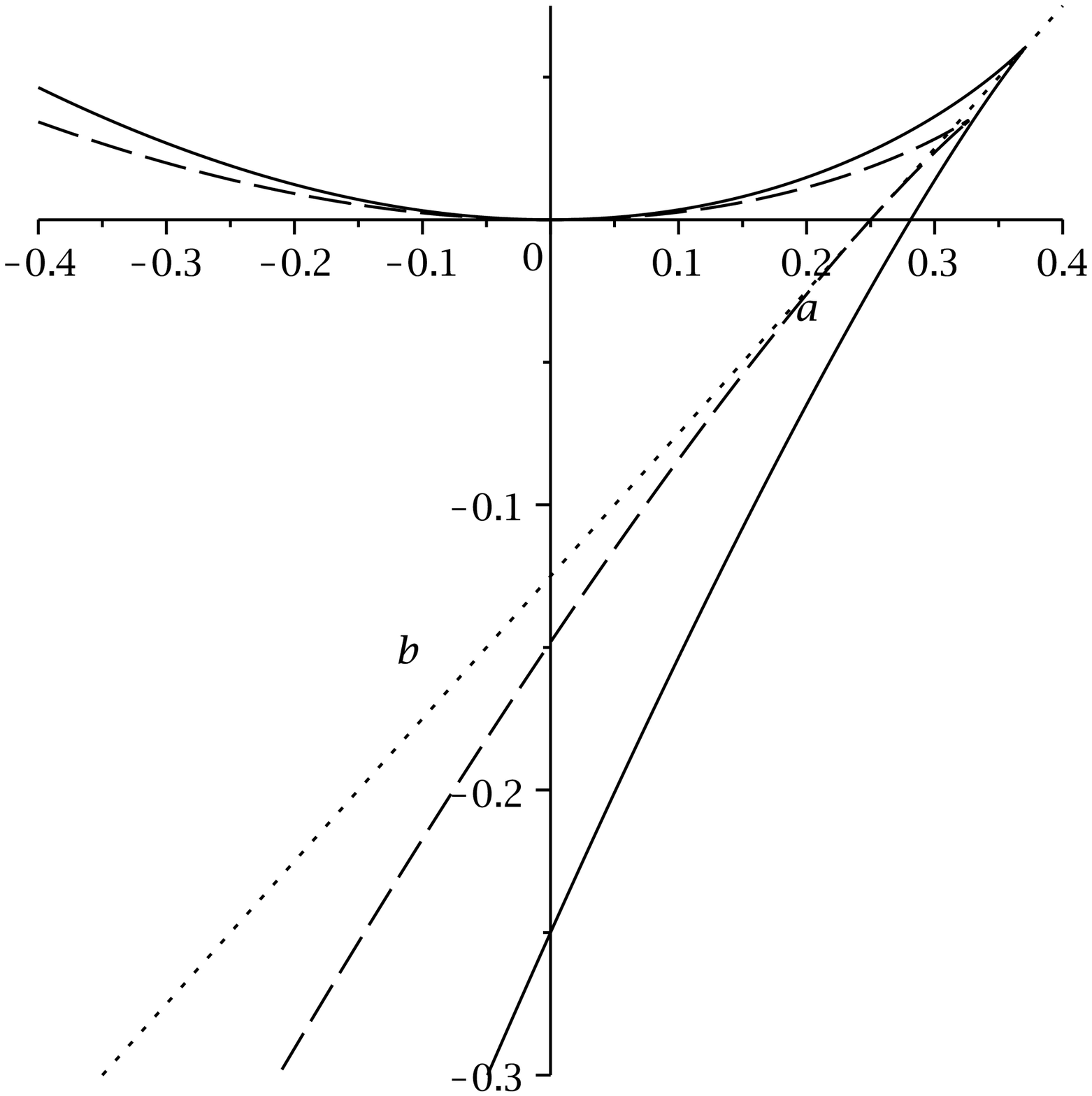}}\hskip1cm \hbox{\includegraphics[scale=0.2]{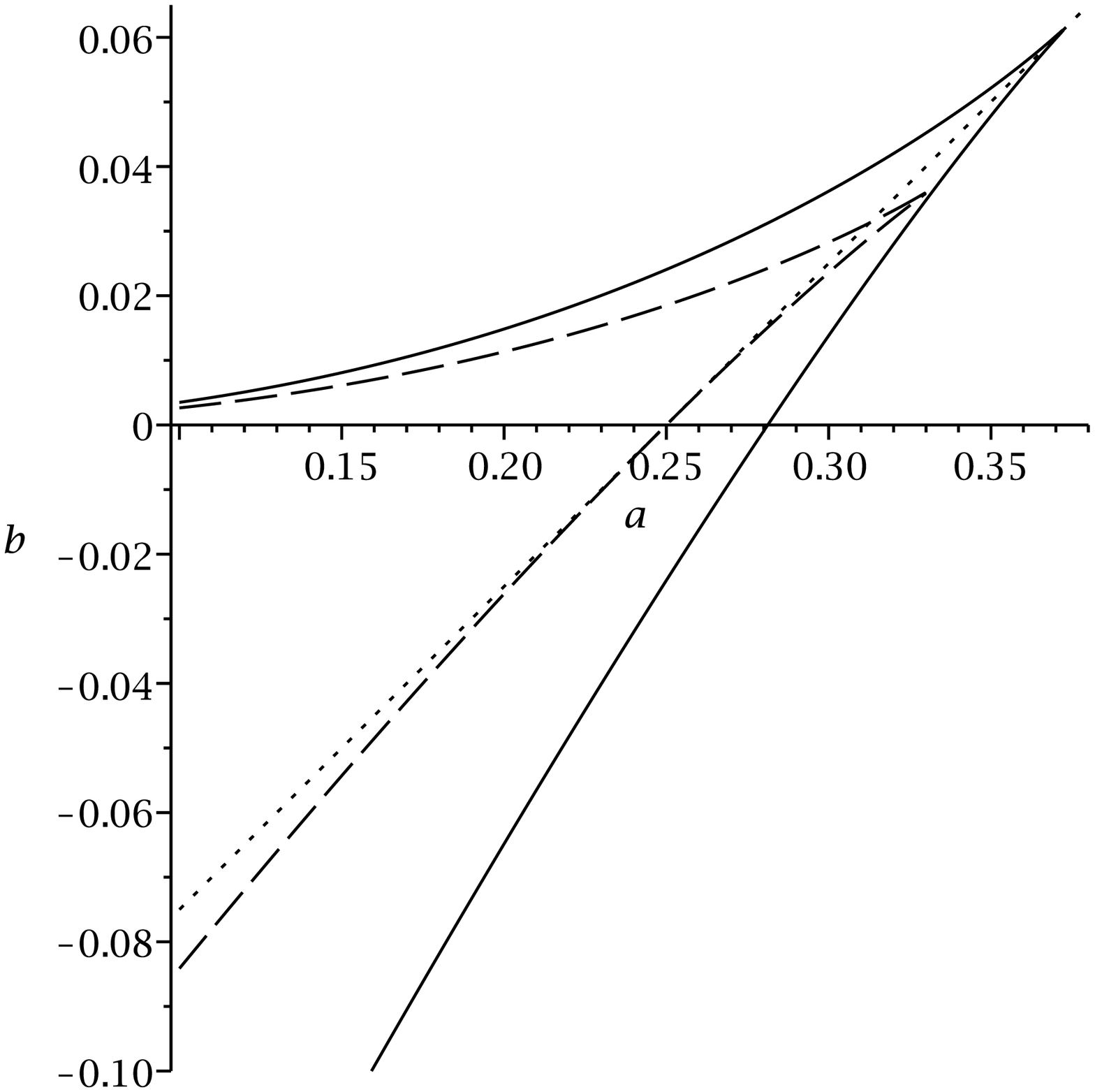}}}
%\centerline{\hbox{\epsfxsize=10cm \epsfbox{k=1234.pdf}}}
    \caption{The projection  of the 
discriminant locus of $x^4+x^3+ax^2+bx+c$ to the plane of parameters  $(a,b)$. (Picture on the right shows the enlarged portion of the projection near the cusp point.)}
\label{n=4B1}
\end{figure}

\subsection{Degree 4} For $n=4,$ we consider the family $P_4:=x^4+x^3+ax^2+bx+c$.  In 
Fig.~\ref{n=4B1}  we show the projection $\tilde{\Phi}$ 
of its discriminant locus $\Phi$  
in the $(a,b)$-plane. (For the other sets their projections 
in $(a,b)$ 
are denoted by the same letters with tilde.) 
By the dashed line we show the set $\Sigma$  for the family $P_3$. 
One has $$\Phi \cap \{ c=0\} =\Sigma \cup \{ b=c=0\}.$$ 
By the solid line we represent the 
projection 

$$\tilde{\Lambda}~:~64a^3-18a^2+54b-216ab+216b^2=0$$ 
of the subset $\Lambda \subset \Phi$ for which the 
polynomial $P_4$ has a real root of multiplicity at least $3$. 
The ordinary cusp 
point of $\tilde{\Lambda}$ is the projection  of the point 
$(3/8,1/16,1/256)$ which defines the polynomial $x^4+x^3+3x^2/8+x/16+1/256=(x+1/4)^4~$ 
to the plane  $(a,b)$. 

\begin{figure}[htbp]
%\includegraphics[scale=0.5]{parthetanegfirstfour.eps}
%\centerline{\hbox{\includegraphics[scale=0.7]{parthetanegfirstfour.eps}}}
\centerline{\hbox{\includegraphics[scale=0.2]{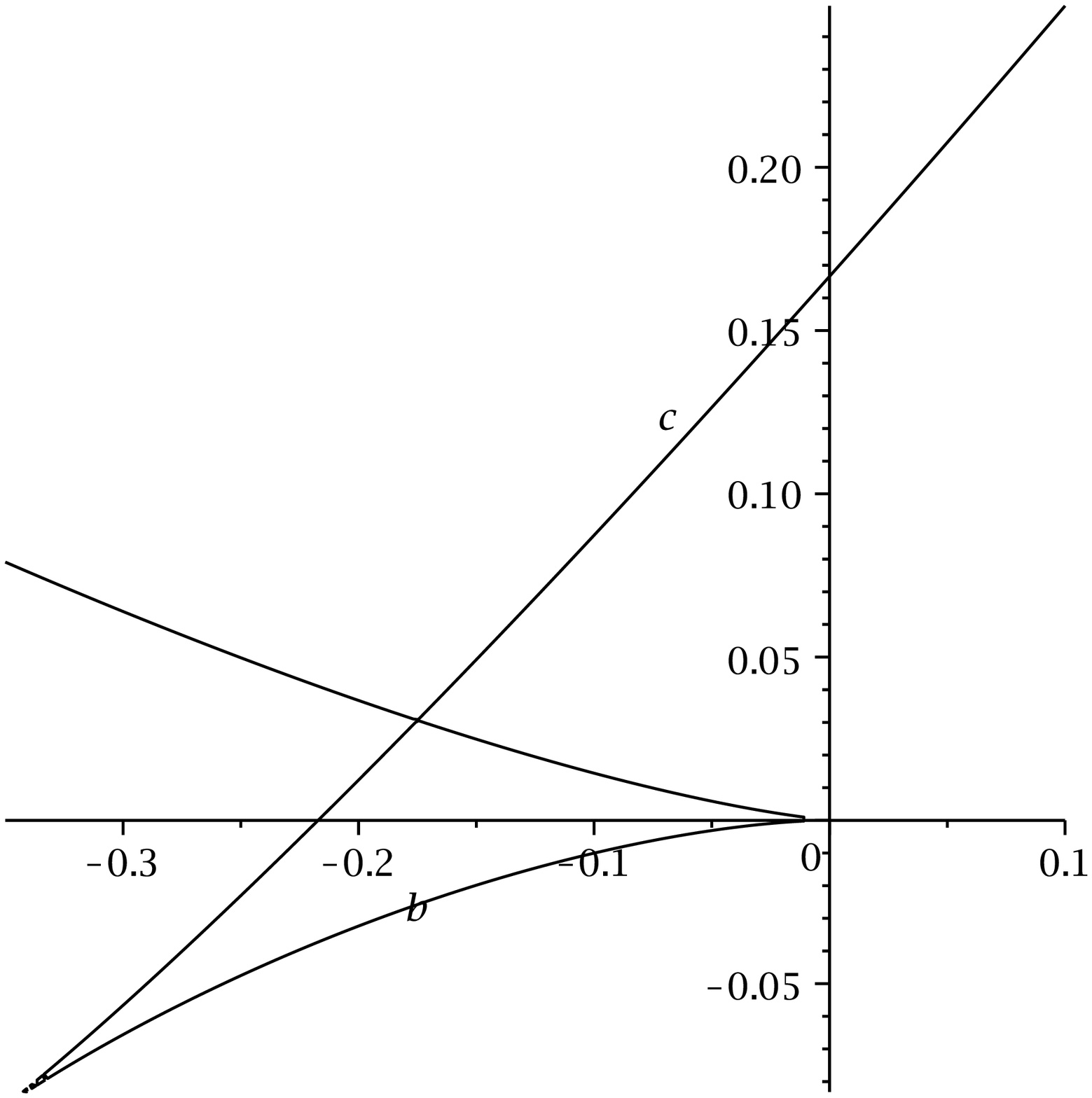}}\hskip1cm \hbox{\includegraphics[scale=0.2]{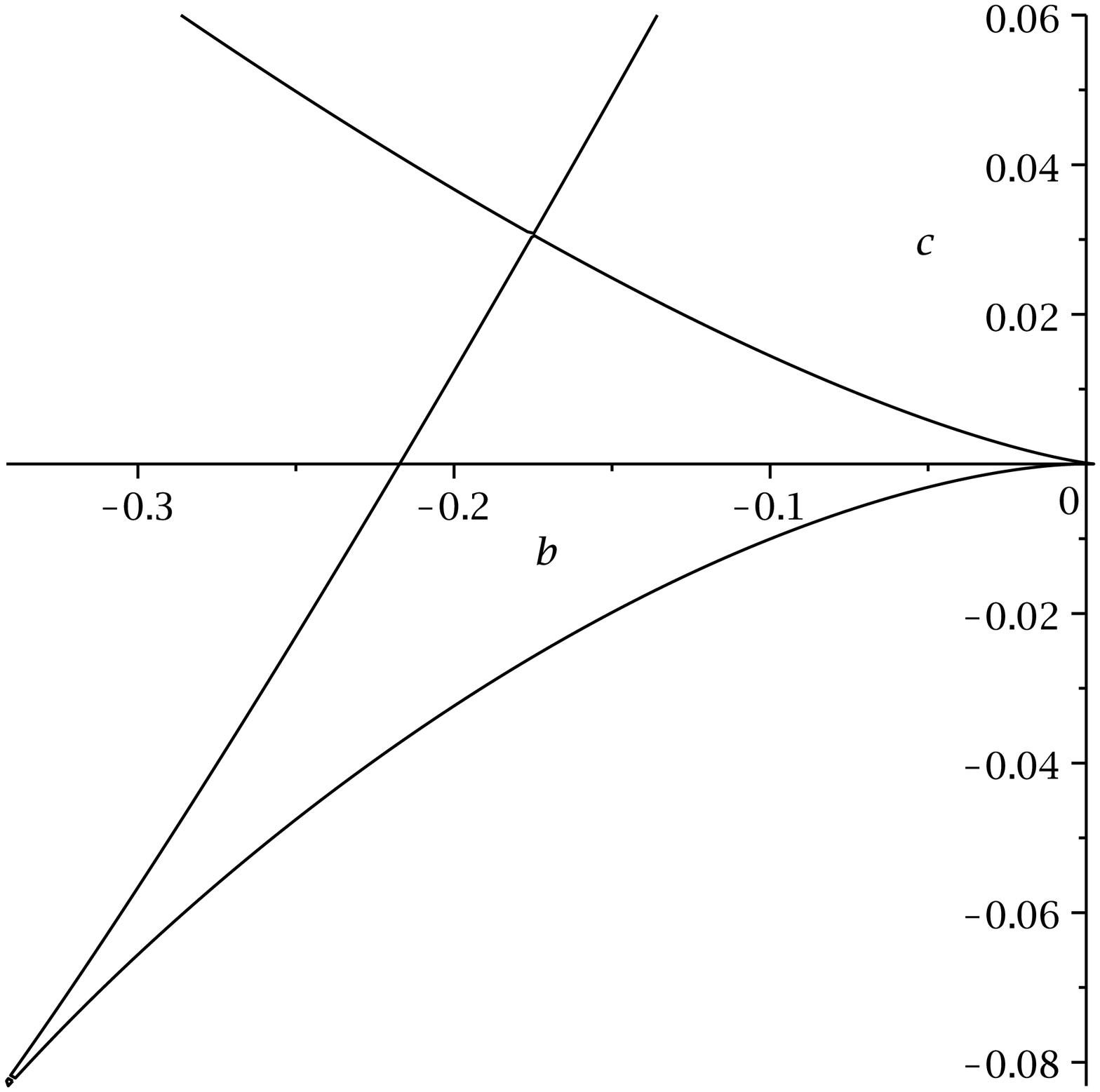}}}
\vskip1cm
\centerline{\hbox{\includegraphics[scale=0.2]{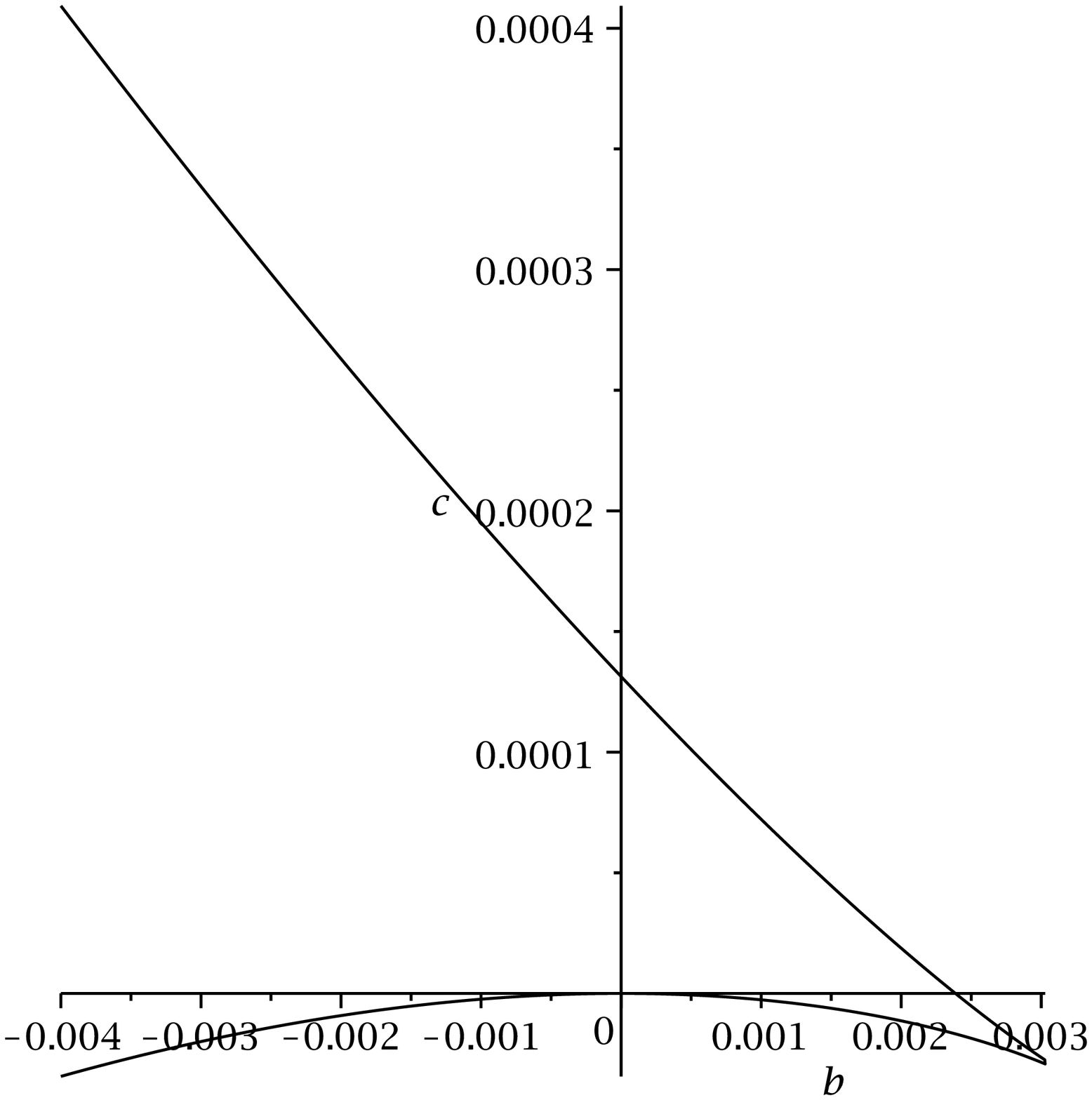}}\hskip1cm \hbox{\includegraphics[scale=0.2]{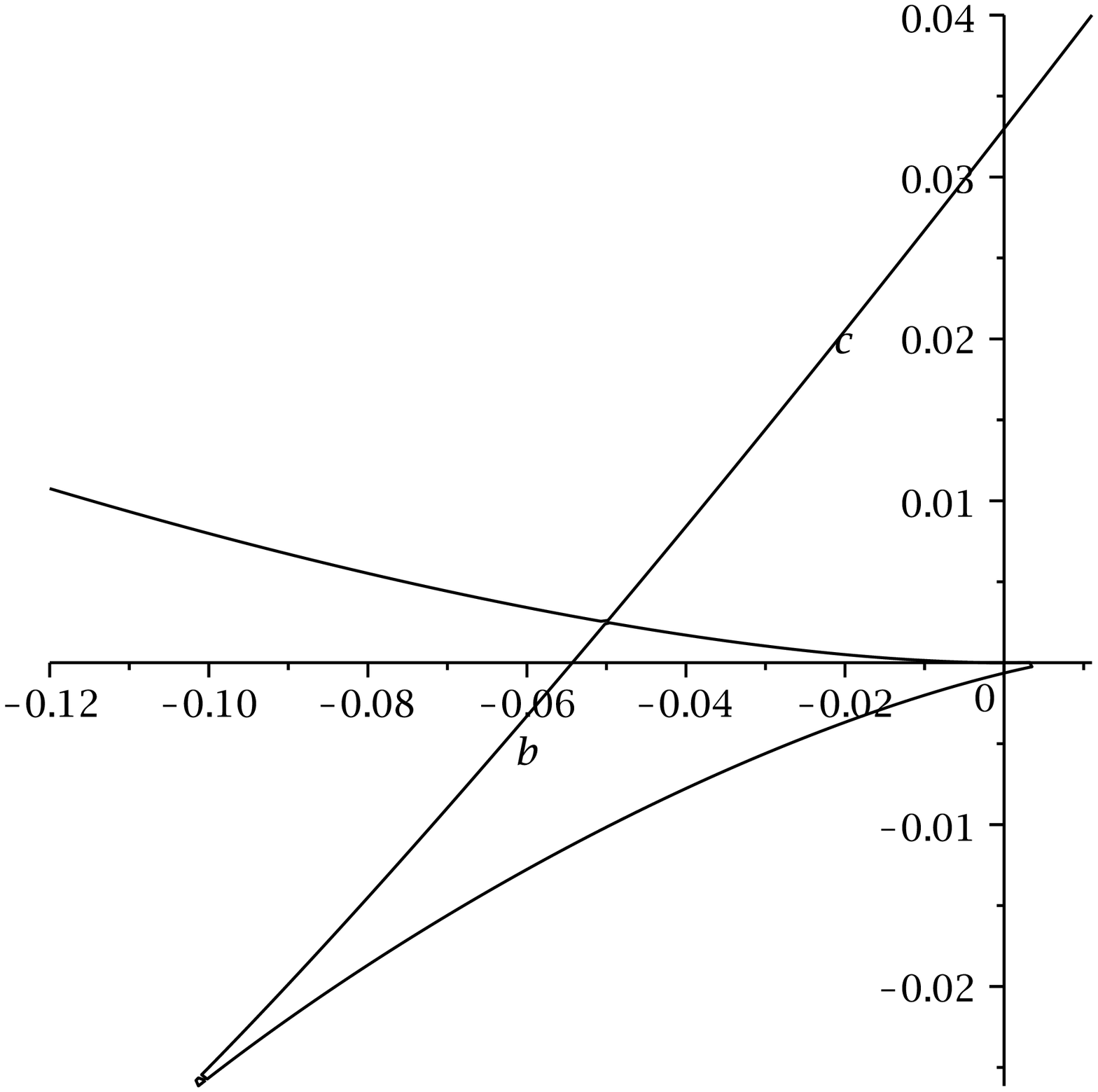}}}
\vskip1cm
\centerline{\hbox{\includegraphics[scale=0.2]{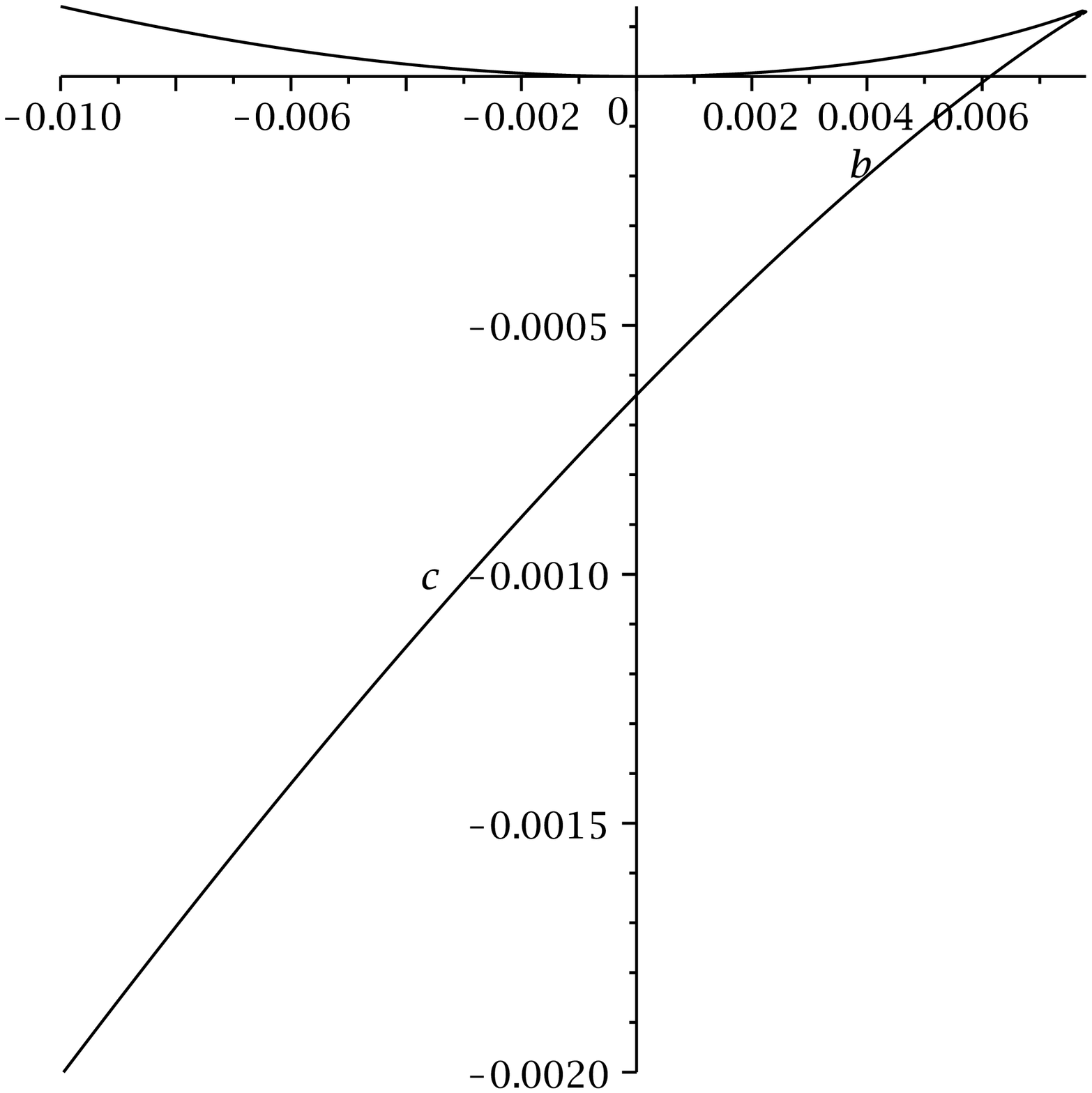}}\hskip1cm \hbox{\includegraphics[scale=0.2]{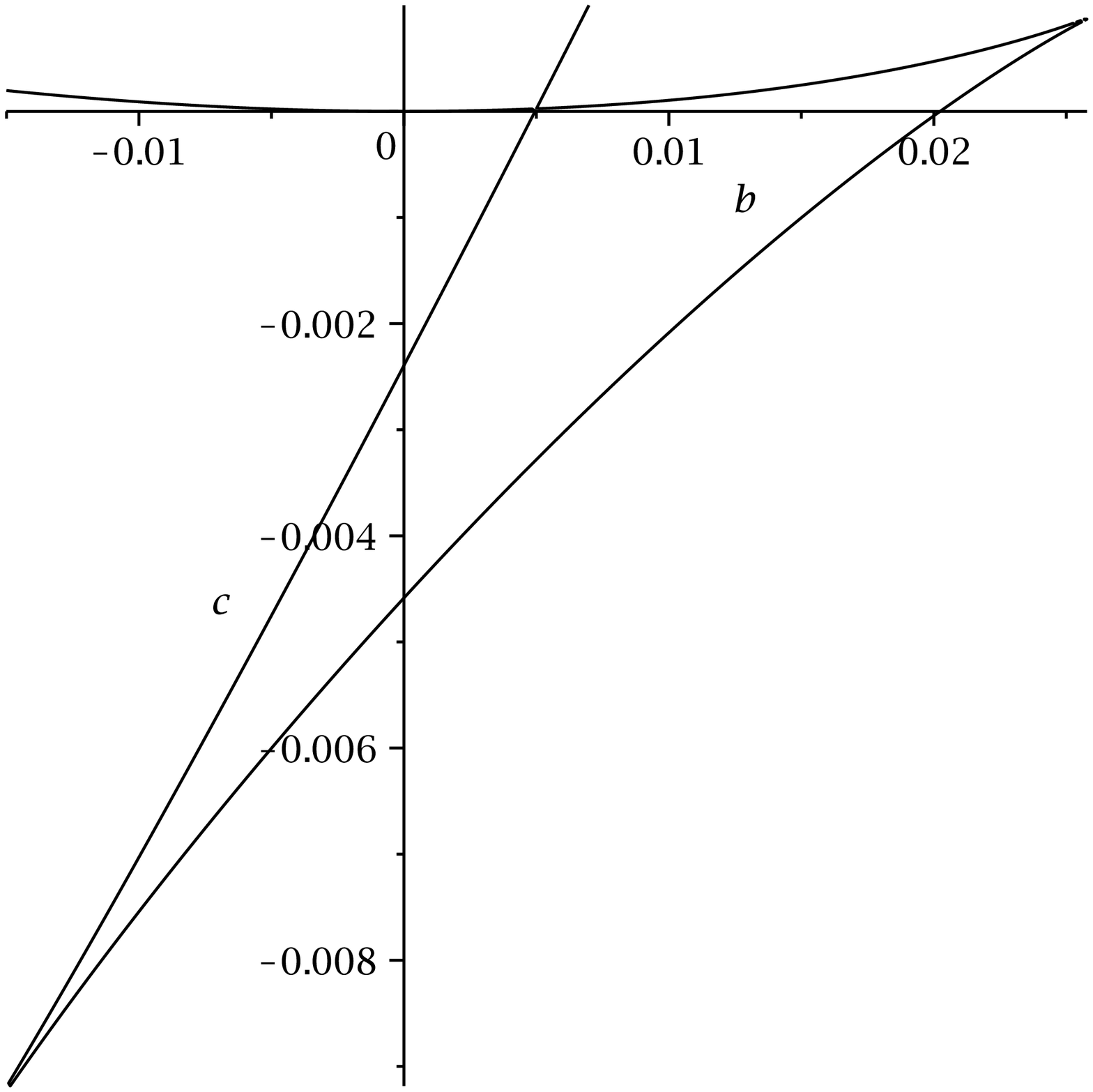}}}
\vskip1cm
%\centerline{\hbox{\epsfxsize=10cm \epsfbox{k=1234.pdf}}}
    \caption{Intersections of the 
discriminant locus of $x^4+x^3+ax^2+bx+c$ with the planes $a=-0.1$ (the first three pictures); $a=0.15$ (the fourth and the fifth pictures); and $a= 0.26$ (the last picture).}
\label{n=4Bsecm01}
\end{figure}

%\begin{figure}[htbp]
%\includegraphics[scale=0.5]{parthetanegfirstfour.eps}
%\centerline{\hbox{\includegraphics[scale=0.7]{parthetanegfirstfour.eps}}}
%\centerline{\hbox{\includegraphics[scale=0.4]{n=4B2.eps}}}
%\centerline{\hbox{\epsfxsize=10cm \epsfbox{k=1234.pdf}}}
%    \caption{The projection in $(a,b)$ of the 
%discriminant set of $x^4+x^3+ax^2+bx+c$.}
%\label{n=4B2}
%\end{figure}

At this point the set $\Phi$ has a swallowtail singularity, see e.g. \cite{Ar}.  
On the upper arc of $\Lambda$ the polynomial $P_4$ has one triple root to the 
right and a simple one to the left (and vice versa for the lower arc). 
The upper arc of $\tilde{\Lambda}$ has an ordinary tangency 
to the $a$-axis at the origin. 
Along the curve $\Lambda$ the intersections of the hypersurface $\Phi$ with 
planes transversal to $\Lambda$ have cusp points. 
 
The cusp point of $\Sigma$ belongs to $\Lambda$. At this point $\Lambda$ 
intersects the $(a,b)$-plane. The tangent line 
$\tilde{L}:b=a/2-1/8$ to $\tilde{\Lambda}$ at its cusp at $(3/8,1/16)$ 
 is tangent to the curve $\Sigma$ at 
$(1/4,0)$. ($\tilde{L}$ is shown by the dotted line.) The set $L$ corresponds to polynomials having two double roots. 
For $a<3/8,$ these roots are real, and for $a>3/8,$ they are complex conjugate. 
The curve $L$ is tangent to the $(a,b)$-plane at the point $(1/4,0,0)$. It 
belongs to the half-space $\{ c\geq 0\}$. 

Now we consider the intersections of $\Phi$ with the planes parallel to the 
$(b,c)$-plane. For $a<3/8,$ they  have two ordinary cusps (which are 
the points of $\Lambda$) and a transversal self-intersection point (which belongs to $L$). The first three pictures in Fig.~\ref{n=4Bsecm01}  
%, \ref{n=4Bsecm01bis} and \ref{n=4Bsecm01bisbis} 
show this intersection with 
the plane $a=-0.1$ in different scales. The curves are tangent to the $a$-axis. Inside the 
curvilinear triangle (denoted by $H_4$) 
the polynomial has four distinct real roots. In the domain $H_2$ 
which surrounds $H_4$, the polynomial $P_4$ has two distinct real roots 
and a complex conjugate pair. 
In the domain $H_0$ above the self-intersection point it has two complex conjugate 
pairs. These domains are defined in the same way for all $a<3/8$. For 
$a>3/8$, the domain $H_4$ does not exist.

\begin{figure}[htbp]
%\includegraphics[scale=0.5]{parthetanegfirstfour.eps}
%\centerline{\hbox{\includegraphics[scale=0.7]{parthetanegfirstfour.eps}}}
\centerline{\hbox{\includegraphics[scale=0.2]{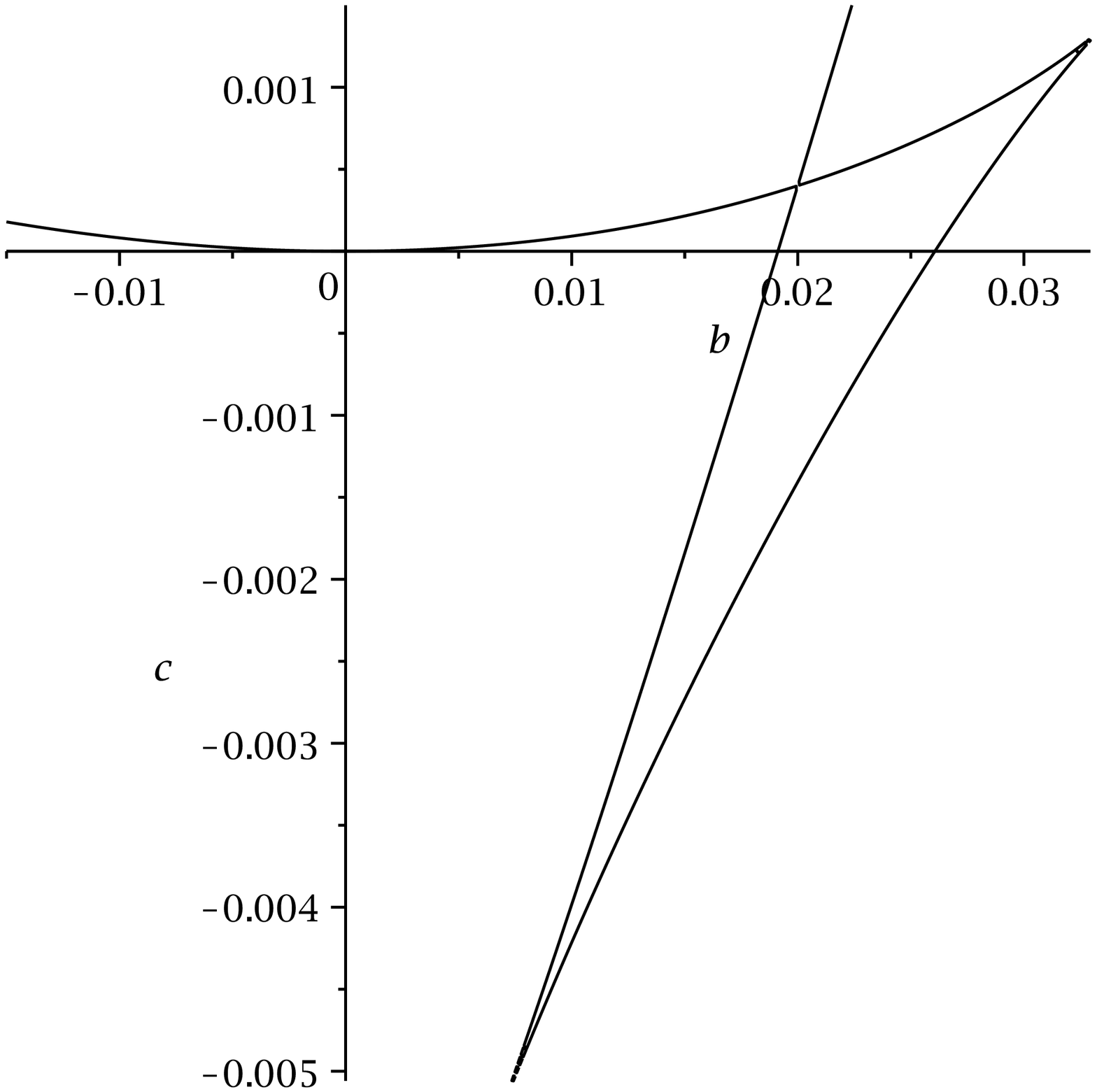}}\hskip1cm\hbox{\includegraphics[scale=0.2]{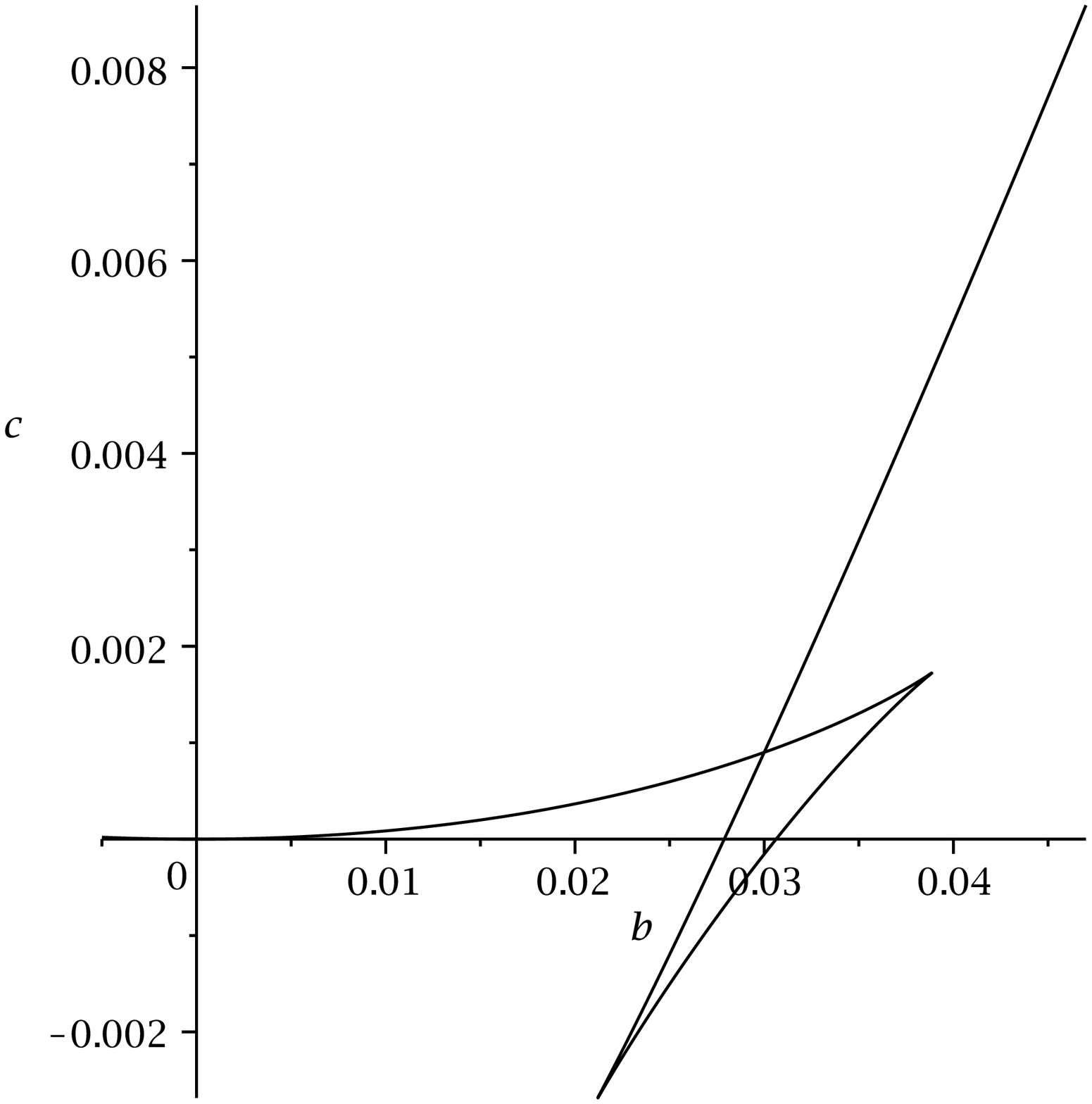}}}
\vskip1cm
\centerline{\hbox{\includegraphics[scale=0.2]{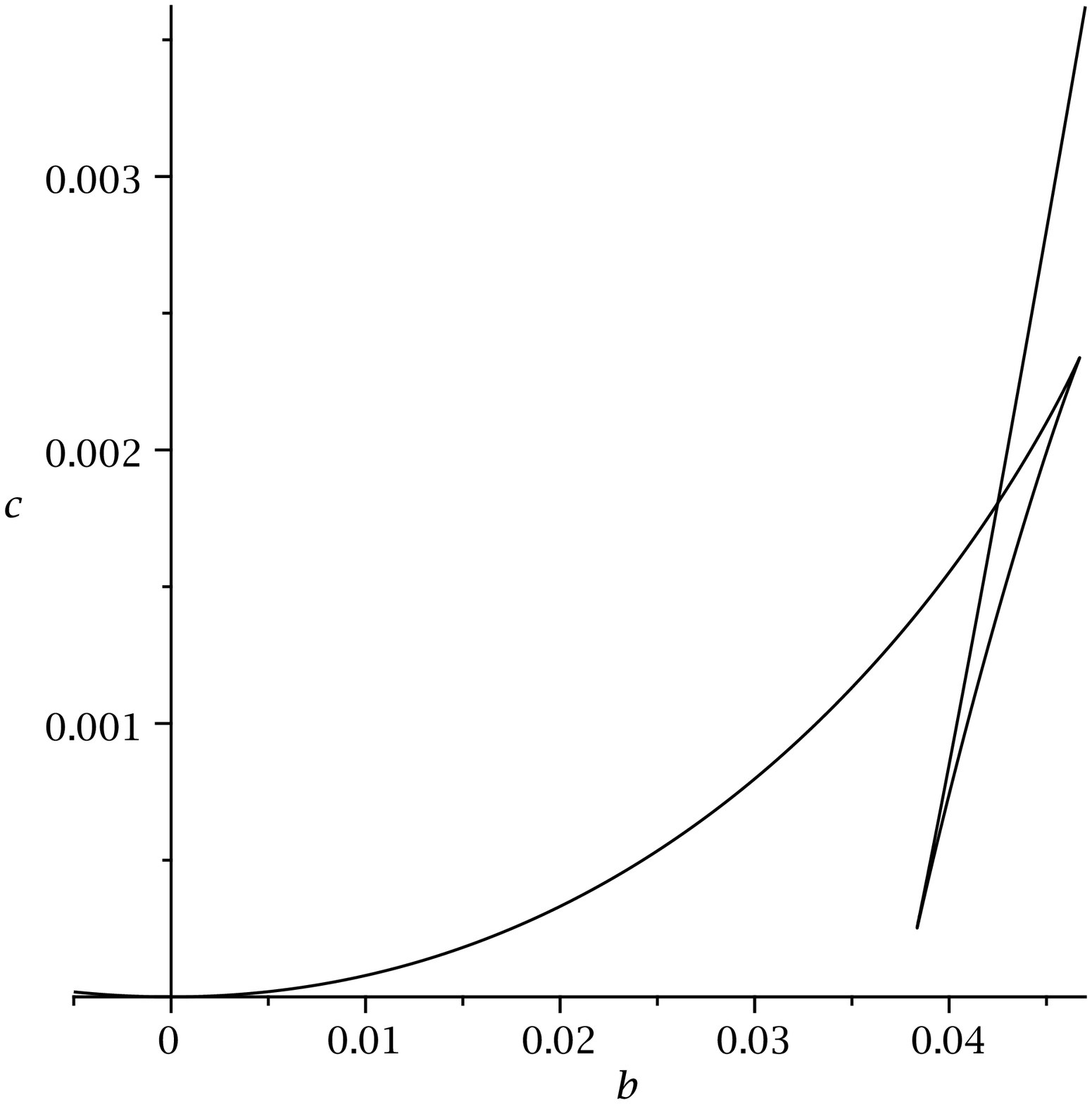}}\hskip1cm \hbox{\includegraphics[scale=0.2]{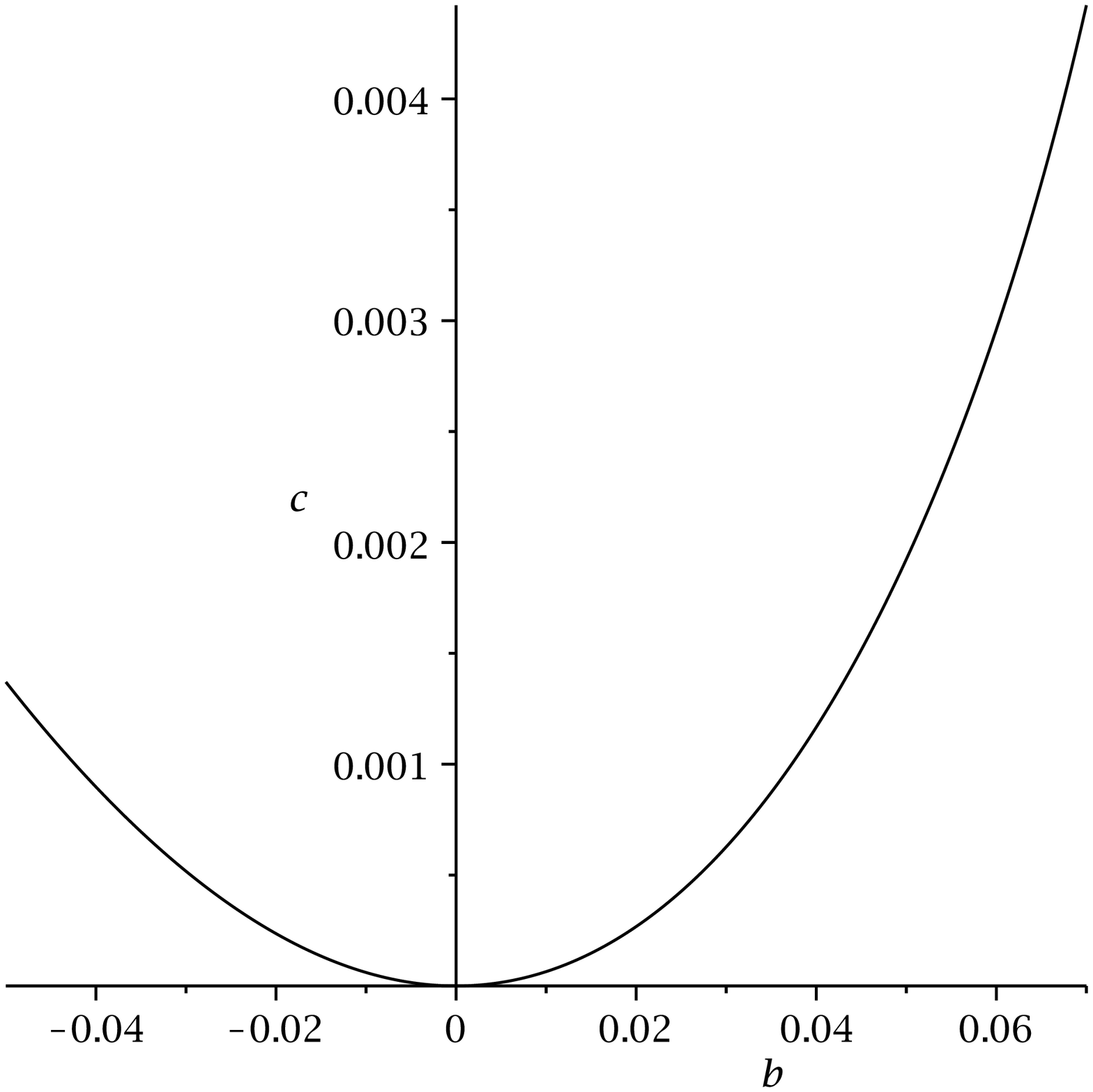}}}
\vskip0.5cm
%\centerline{\hbox{\epsfxsize=10cm \epsfbox{k=1234.pdf}}}
    \caption{The intersection of the 
discriminant locus of $x^4+x^3+ax^2+bx+c$ with the planes $a=0.29; 0.31; 0.335; 0.4$.}% (left to right and down).}
\label{n=4Bsec029}
\end{figure} 

%\begin{figure}[htbp]
%\includegraphics[scale=0.5]{parthetanegfirstfour.eps}
%\centerline{\hbox{\includegraphics[scale=0.7]{parthetanegfirstfour.eps}}}
%\centerline{\hbox{\includegraphics[scale=0.4]{n=4Bsecm01bis.eps}}}
%\centerline{\hbox{\epsfxsize=10cm \epsfbox{k=1234.pdf}}}
%    \caption{The intersection of the 
%discriminant set of $x^4+x^3+ax^2+bx+c$ with $\{ a=-0.1\}$.}
%\label{n=4Bsecm01bis}
%\end{figure}

%\begin{figure}[htbp]
%\includegraphics[scale=0.5]{parthetanegfirstfour.eps}
%\centerline{\hbox{\includegraphics[scale=0.7]{parthetanegfirstfour.eps}}}
%\centerline{\hbox{\includegraphics[scale=0.35]{n=4Bsecm01bisbis.eps}}\hskip1cm \hbox{\includegraphics[scale=0.35]{n=4Bsec015.eps}}}
%\centerline{\hbox{\epsfxsize=10cm \epsfbox{k=1234.pdf}}}
 %   \caption{The intersection of the 
%discriminant set of $x^4+x^3+ax^2+bx+c$ with $\{ a=-0.1\}$.}
%\label{n=4Bsecm01bisbis}
%\end{figure}

The set $\Phi \cap \{ a<0,b<0,c>0\}$ divides the set $\{ a<0,b<0,c>0\}$ into 
four sectors, see the first picture in Fig.~\ref{n=4Bsecm01}. 
The intersection $\{ a<0,b<0,c>0\} \cap H_2$ consists of two contractible 
components. They correspond to the two cases $(0,2)$ (the right sector, 
bordering $\{ a<0,b>0,c>0\}$) and $(2,0)$ (the left sector) realizable 
with the sign pattern $(+,+,-,-,+)$. The other two cases realizable in 
$\{ a<0,b<0,c>0\}$ are $(2,2)$ (the sector below) and $(0,0)$ (the sector above). 

For $a<0$, $b>0$, $c>0$, and when the polynomial $P_4$ belongs 
respectively to $H_4$, $H_2$ or $H_0$, it realizes the cases  
$(2,2)$, $(0,2)$ and $(0,0)$. The set  
$\{ a<0,b>0,c>0\} \cap H_2$ is contractible, so only one of the cases 
$(0,2)$ and $(2,0)$ (namely, $(0,2)$) is realizable with the sign pattern 
$(+,+,-,+,+)$ (see the first picture in Fig.~\ref{n=4Bsecm01}).

In $\{ a<0,b<0,c<0\}$ one can realize the cases $(1,3)$ and $(1,1)$. They 
correspond to the domains $\{ a<0,b<0,c<0\} \cap H_4$ (the curvilinear 
triangle) and 
$\{ a<0,b<0,c<0\} \cap H_2$ (its complement). 

In $\{ a<0,b>0,c<0\}$ one  can similarly realize the cases $(3,1)$ (the curvilinear 
triangle) and $(1,1)$ (its complement).

On the fourth and fifth pictures in Fig.~\ref{n=4Bsecm01}
%Fig.~\ref{n=4Bsec015} and \ref{n=4Bsec015bis} 
we present 
the intersection of $\Phi$ with the 
plane $\{ a=0.15\}$. The figures are quite similar to  the first three 
pictures in Fig.~\ref{n=4Bsecm01}, 
%\ref{n=4Bsecm01bis} and \ref{n=4Bsecm01bisbis} 
and the realizable pairs are the same with one 
exception. Namely,  for $a>0$, $b>0$, $c>0$ in the domain $H_4$ it is the pair $(0,4)$ 
which is realized. And, clearly, the third 
component of the sign patterns changes from $-$ to $+$.

%\begin{figure}[htbp]
%\includegraphics[scale=0.5]{parthetanegfirstfour.eps}
%\centerline{\hbox{\includegraphics[scale=0.7]{parthetanegfirstfour.eps}}}
%\centerline{\hbox{\includegraphics[scale=0.4]{n=4Bsec015.eps}}}
%\centerline{\hbox{\epsfxsize=10cm \epsfbox{k=1234.pdf}}}
 %   \caption{The intersection of the 
%discriminant set of $x^4+x^3+ax^2+bx+c$ with $\{ a=0.15\}$.}
%\label{n=4Bsec015}
%\end{figure}

%\begin{figure}[htbp]
%\includegraphics[scale=0.5]{parthetanegfirstfour.eps}
%\centerline{\hbox{\includegraphics[scale=0.7]{parthetanegfirstfour.eps}}}
%\centerline{\hbox{\includegraphics[scale=0.35]{n=4Bsec015bis.eps}}\hskip1cm \hbox{\includegraphics[scale=0.35]{n=4Bsec026.eps}}}
%\centerline{\hbox{\epsfxsize=10cm \epsfbox{k=1234.pdf}}}
%    \caption{The intersection of the 
%discriminant set of $x^4+x^3+ax^2+bx+c$ with $\{ a=0.15\}$.}
%\label{n=4Bsec015bis}
%\end{figure} 
The intersections of $\Phi$ with the planes $\{ a=0.26\}$, $\{ a=0.29\}$, 
$\{ a=0.31\}$  
and $\{ a=0.335\}$ are shown on the last picture in Fig.~\ref{n=4Bsecm01} 
and in Fig.~\ref{n=4Bsec029}. For $a_0>0.375$, the 
intersections of $\Phi$ with the planes $\{ a=a_0\}$  
resemble the lower right picture in Fig.~\ref{n=4Bsec029}.

\section{Final Remarks} 

%Above we mainly discussed the question  which pairs $(pos,neg)$ of the numbers of positive and negative roots satisfying the obvious compatibility conditions are realised by polynomials with a given sign pattern.  
The following  important and closely related to the main topic of the present paper questions remained unaddressed above.  

\begin{problem}
Is the set of all polynomials realizing a given pair $(pos,neg)$ and having a sign pattern $\bsi$ path-connected (if non-empty)?  
\end{problem}

Given a real polynomial $p$ of degree $d$ with all non-vanishing coefficients, consider  the sequence of pairs 
$$\{(pos_0(p),neg_0(p)), (pos_1(p),neg_1(p)), (pos_2(p),neg_2(p)), \dots, (pos_{d-1}(p),neg_{d-1}(p))\},$$ where
$(pos_j(p),neg_j(p))$ is the numbers of positive and negative roots of $p^{(j)}$ respectively.  Observe that if one knows the above sequence of pairs, then one knows the sign pattern of a polynomial $p$ which is assumed to be monic.  Additionally it is easy to construct examples when the converse fails. 

\begin{problem}
Which sequences of pairs are realizable by real polynomials of degree $d$ with all non-vanishing coefficients? 
\end{problem}
Notice that a similar problem for the sequence of pairs of real roots (without division into positive and negative) was considered in \cite{Ko}. 
One can find easily examples of non-realizable sequences $\{ (pos_j(p),neg_j(p))\} _{j=0}^{d-1}$. E.~g. for $d=4$ this is the sequence 
$(2,0)$, $(2,1)$, $(1,1)$, $(0,1)$. Indeed, the sign pattern must be $(+,+,-,+,+)$ about which we know that it is not realizable 
with the pair $(2,0)$. However it is not self-evident that all non-realizable sequences are obtained in this way.

\medskip
Our final question is as follows. 

\begin{problem}
Is the set of all polynomials realizing a given sequence of pairs  as above path-connected (if non-empty)?  
\end{problem}

\end{document}